\documentclass[12pt]{amsart}
\usepackage{amscd}
\usepackage{amsmath}
\usepackage{amssymb}
\usepackage{units}
\usepackage[all]{xy}
\def\frk{\frak}               

\def\mm{{\frk m}}

\def\Phi{{\frk n}}
\def\Phi{{\frk N}}
\def\opn#1#2{\def#1{\operatorname{#2}}} 
\opn\chara{char} \opn\length{\ell} \opn\pd{pd} \opn\rk{rk}
\opn\projdim{proj\,dim} \opn\injdim{inj\,dim} \opn\rank{rank}
\opn\depth{depth} \opn\sdepth{sdepth} \opn\fdepth{fdepth}
\opn\grade{grade} \opn\height{height} \opn\embdim{emb\,dim}
\opn\codim{codim}  \opn\min{min} \opn\max{max}

\opn\Tr{Tr} \opn\bigrank{big\,rank}
\opn\superheight{superheight}\opn\lcm{lcm}
\opn\trdeg{tr\,deg}
\opn\reg{reg} \opn\lreg{lreg} \opn\ini{in} \opn\lpd{lpd}
\opn\size{size}
\opn\div{div} \opn\Div{Div} \opn\cl{cl} \opn\Cl{Cl}
\opn\Spec{Spec} \opn\Supp{Supp} \opn\supp{supp} \opn\Sing{Sing}
\opn\Ass{Ass} \opn\Min{Min}
\opn\Ann{Ann} \opn\Rad{Rad} \opn\Soc{Soc}
\opn\Im{Im} \opn\Ker{Ker} \opn\Coker{Coker} \opn\Am{Am}
\opn\Hom{Hom} \opn\Tor{Tor} \opn\Ext{Ext} \opn\End{End}
\opn\Aut{Aut} \opn\id{id}  \opn\deg{deg}

\opn\nat{nat}
\opn\pff{pf}
\opn\Pf{Pf} \opn\GL{GL} \opn\SL{SL} \opn\mod{mod} \opn\ord{ord}
\opn\Gin{Gin} \opn\Hilb{Hilb}
\opn\aff{aff} \opn\con{conv} \opn\relint{relint} \opn\st{st}
\opn\lk{lk} \opn\cn{cn} \opn\core{core} \opn\vol{vol}
\opn\link{link} \opn\star{star}
\opn\gr{gr}

\def\pot#1#2{#1[\kern-0.28ex[#2]\kern-0.28ex]}

\opn\dirlim{\underrightarrow{\lim}}
\opn\inivlim{\underleftarrow{\lim}}
\let\to=\rightarrow

\def\Implies{\ifmmode\Longrightarrow \else
        \unskip${}\Longrightarrow{}$\ignorespaces\fi}
\def\implies{\ifmmode\Rightarrow \else
        \unskip${}\Rightarrow{}$\ignorespaces\fi}
\def\iff{\ifmmode\Longleftrightarrow \else
        \unskip${}\Longleftrightarrow{}$\ignorespaces\fi}

\let\:=\colon
\newtheorem{Theorem}{Theorem}[]
\newtheorem{Lemma}[Theorem]{Lemma}
\newtheorem{Corollary}[Theorem]{Corollary}
\newtheorem{Proposition}[Theorem]{Proposition}
\newtheorem{Remark}[Theorem]{Remark}

\let\epsilon\varepsilon
\let\phi=\varphi
\let\kappa=\varkappa
\def\qed{\ifhmode\textqed\fi
      \ifmmode\ifinner\quad\qedsymbol\else\dispqed\fi\fi}
\def\textqed{\unskip\nobreak\penalty50
       \hskip2em\hbox{}\nobreak\hfil\qedsymbol
       \parfillskip=0pt \finalhyphendemerits=0}
\def\dispqed{\rlap{\qquad\qedsymbol}}

\opn\dis{dis}
\def\pnt{{\raise0.5mm\hbox{\large\bf.}}}

\opn\Lex{Lex}

\begin{document}
\title{  Nested Artin approximation }

\author{ Dorin Popescu }
\thanks{The  support from the the project  ID-PCE-2011-1023, granted by the Romanian National Authority for Scientific Research, CNCS - UEFISCDI   is gratefully acknowledged. }

\address{Dorin Popescu, Simion Stoilow Institute of Mathematics of the Romanian Academy, Research unit 5,
University of Bucharest, P.O.Box 1-764, Bucharest 014700, Romania}
\email{dorin.popescu@imar.ro}

\maketitle

\begin{abstract} A short proof of the linear nested Artin approximation property of the algebraic power series rings is given here.

\end{abstract}

\vskip 0.5 cm

\section*{Introduction}

The solution of an old problem (see for instance \cite{K}), the
so-called  nested Artin approximation property, is given in the following theorem.
\begin{Theorem}[\cite{P}, {\cite[Theorem 3.6]{P1}}] \label{nes}
 Let $k$ be a field,  $ A=k\langle x\rangle$, $x=(x_1,\ldots,x_n)$ the algebraic power series over $k$,  $f=(f_1,\ldots,f_s)\in A[Y]^s$, $Y=(Y_1,\ldots,Y_m)$  and  $0<r_1\leq \ldots \leq r_m\leq n$, $c$ be some non-negative integers. Suppose that $f$ has a solution ${\hat y}=({\hat y}_1,\ldots,{\hat y}_m)$ in   $ k[[x]] $ such that ${\hat y}_i\in k[[x_1,\ldots,x_{r_i}]]$ for all $1\leq i\leq m$ (that is a so-called nested formal solution). Then there exists a solution $y=(y_1,\ldots,y_m)$  of $f$ in $A$ such that $y_i\in k\langle x_1,\ldots,x_{r_i}\rangle$ for all $1\leq i\leq m$ and $y\equiv{\hat y} \ \ \mbox{mod}\ \ (x)^ck[[x]]$.
\end{Theorem}
The proof relies on an idea of Kurke from 1972 and the Artin approximation property of rings of type $k[[u]]\langle x\rangle$ (see \cite{P}, \cite{S}, \cite{P1}). When $f$ is linear, interesting relations with other problems and a description of many results on this topic are nicely explained in \cite{Rond1}.
Also a proof of the above theorem in the linear case
 can be found in \cite{Rond1} using just the classical Artin approximation property of $A$ (see \cite{A}). Unfortunately, we had some difficulties in reading  \cite{Rond1}, but finally we  noticed a shorter proof using  mainly the same ideas. This proof is the content of the present note.
 
 We owe thanks to G. Rond who noticed a gap in a previous version of this  note.

\section{Linear nested Artin approximation  property}

We start recalling   \cite[Lemma 9.2]{Rond} with  a simplified proof.

\begin{Lemma}\label{am} Let $(A,\mm)$ be a complete normal local domain, $u=(u_1,\dots,u_n)$, $v=(v_1,\ldots,v_m)$, $B=A[[u]]\langle v\rangle $ be the algebraic closure of $A[[u]][v]$ in $A[[u,v]]$ and $f\in B$. Then there exists $g$ in the algebraic closure $A\langle v, Z\rangle $ of $A[v,Z]$, $Z=(Z_1,\ldots,Z_s)$ in $A[[v,Z]]$
 for some $s\in \bf N$ and ${\hat z}\in A[[u]]^s$ such that $f=g(\hat z)$.
 \end{Lemma}
 \begin{proof} Changing $f$ by $f-f(u=0)$ we may assume that $f\in (u)$. Note that $B$ is the Henselization of $C=A[[u]][v]_{(\mm,u,v)}$  by \cite{Ra} and so there exists some etale neighborhood of $C$ containing $f$. Using for example \cite[Theorem 2.5]{S} there exists a monic polynomial $F$ in $X$ over $A[[u]][v]$ such that $F(f)=0$ and $F'(f)\not \in (\mm,u,v)$, let us say $F=\sum_{i,j}F_{ij}v^iX^j$ for some $F_{ij}\in A[[u]]$.

 Set ${\hat z}_{ij}=F_{ij}-F_{ij}(u=0)\in (u)A[[u]]$, ${\hat z}=({\hat z}_{ij})$ and $G=\sum_{ij}(F_{ij}(u=0)+Z_{ij})v^iX^j$ for some new variables $Z=(Z_{ij})$. We  have $G({\hat z})=F$. Set $G'=\partial G/\partial X$. As
 $$G(Z=0)=G({\hat z}(u=0))\equiv F(f)=0\ \mbox{ modulo}\ (u),$$
  $$G'(Z=0)=G'({\hat z}(u=0))\equiv F'(f)\not \equiv 0 \ \mbox{modulo}\ (\mm,u,v)$$
   we get $G(X=0)\equiv 0$, $G'(X=0)\not \equiv 0$ modulo $(\mm,v,Z)$. By the Implicit Function Theorem there exists $g\in (\mm,v,Z) A\langle v,Z\rangle$
such that $G(g)=0$. It follows that $G(g({\hat z}))=0$. But $F=G({\hat z})=0$ has just a solution $X=f$ in $(\mm,u,v)B$ by the Implicit Function Theorem and so $f=g({\hat z})$.
\hfill\ \end{proof}
\begin{Lemma} \label{la} Let $(A,\mm)$ be a Noetherian local ring, $f\in A[U]$, $U=(U_1,\ldots,U_s)$ a linear system of polynomial equations, $c\in \bf N$ and ${\hat u}$ a solution of $f$ in the completion $\hat A$ of $A$. Then there exists a solution $u\in A^s$ of $f$ such that $u\equiv {\hat u}\ \mbox{modulo}\ \mm^c\hat A$.
\end{Lemma}
\begin{proof} Let $B=A[U]/(f)$ and $h:B\to {\hat A}$ be the map given by $U\to {\hat u}$. By \cite[Lemma 4.2]{P} (or \cite[Proposition 36]{P2}) $h$ factors through a polynomial algebra $A[Z]$, $Z=(Z_1,\ldots,Z_s)$, let us say $h$ is the composite map $B\xrightarrow{t} A[Z]\xrightarrow{g} {\hat A}$. Choose $ z\in A^s$ such that $z\equiv g(Z)\ \mbox{modulo} \ \mm^c\hat A$. Then $u=t(cls\ U)(z)$ is a solution of $f$ in $A$ such that $u\equiv {\hat u}\ \mbox{modulo}\ \mm^c\hat A$.
\hfill\ \end{proof}

\begin{Proposition}  \label{p} Let $k\langle x,y\rangle$, $x=(x_1,\ldots,x_n)$, $y=(y_1,\ldots,y_m)$ be the ring of algebraic power series in $x,y$ over a field $k$ and $M\subset k\langle x,y\rangle^p$ a finitely generated $k\langle x,y\rangle$-submodule. Then
$$ k[[x]] (M\cap k\langle x\rangle^p)=(k[[x,y]]M)\cap k[[x]]^p,$$
equivalently $M\cap k\langle x\rangle^p$ is dense in $(k[[x,y]]M)\cap k[[x]]^p$, that is for all ${\hat v}\in (k[[x,y]]M)\\
\cap k[[x]]^p$ and  $c\in \bf N$ there exists
$v_c\in (M\cap k\langle x\rangle)^p$ such that $v_c\equiv {\hat v}\ \mbox{modulo}\ (x)^ck[[x]]^p$.
 Moreover, if $c\in \bf N$ and   ${\hat v}=\sum_{i=1}^t {\hat u}_i a_i$  for some $a_i\in M$, ${\hat u}_i\in k[[x,y]]$ then there exist $ u_{ic}\in k\langle x,y\rangle$ such that
   $ u_{ic}\equiv {\hat u}_i\  \mbox{modulo} \ (x)^ck[[x,y]] $, $v_c=\sum_{i=1}^t  u_{ic}a_i\in (M\cap k\langle x\rangle^p) $ and  $\hat v$ is the limit of  $(v_c)_c$ in the $(x)$-adic topology.
\end{Proposition}
\begin{proof}  Let  ${\hat v}\in (k[[x,y]]M)\cap k[[x]]^p$, let us say
${\hat v}=\sum_{i=1}^t {\hat u}_ia_i$ for some $a_i\in M$ and  ${\hat u}_i\in k[[x,y]]^p$.
By flatness of $k[[x]]\langle y\rangle\subset k[[x,y]]$ we see that there exist ${\tilde u}_i\in k[[x]]\langle y\rangle$ such that ${\hat v}=\sum_{i=1}^t {\tilde u}_ia_i$. Moreover by Lemma \ref{la} we may choose ${\tilde u}_i$ such that ${\tilde u}_i\equiv {\hat u}_i\ \mbox{modulo} \ (x)^ck[[x,y]]$.
Then using Lemma \ref{am} there exist $g_i\in k\langle y,Z\rangle$, $i\in [t]$ for some variables $Z=(Z_1,\ldots,Z_s)$ and ${\hat z}\in k[[x]]^s$ such that ${\tilde u}_i=g_i({\hat z})$. Note that $a_i=\sum_{r\in {\bf N}^m}  a_{ir}y^r$, $g_i=\sum_{r\in {\bf N}^m}  g_{ir}y^r$ with $a_{ir}\in k\langle x\rangle^p$, $g_{ir}\in k\langle Z\rangle$.

Clearly, ${\hat z}, {\hat v}$ is  a solution  in $k[[x]]$  of the system of polynomial equations $ V=\sum_{i=1}^t a_ig_i(Z) $, $V=(V_1,\ldots, V_p)$  if and only if it is a solution of the infinite system of polynomial equations

($*$) $V=\sum_{i=1}^t a_{i0}g_{i0}(Z)$, \ \ \ $\sum_{i=1}^t \sum_{r+r'=e}a_{ir}g_{ir'}(Z)=0$, $e\in {\bf N}^m$, $e\not =0.$

Since $k\langle x,Z,V\rangle$ is Noetherian we see that it is enough to consider in ($*$) only a finite set of  equations, let us say  with $e\leq \omega$ for some $\omega$ high enough. Applying the Artin approximation property of $k\langle x\rangle$ (see \cite{A}) we can find for any $c\in \bf N$ a solution $v_c\in k\langle x\rangle^p$, $z_c\in k\langle x\rangle^s$ of ($*$) such that $v_c\equiv {\hat v}\ \mbox{modulo}\ (x)^ck[[x]]^p$, $z_c\equiv {\hat z}\ \mbox{modulo}\ (x)^ck[[x]]^s$. Then $v_c=\sum_{i=1}^t a_ig_i(z_c)\in M\cap k\langle x\rangle^p $, and $u_{ic}=g_i(z_{ic})\in k\langle x,y\rangle$ satisfies $u_{ic}\equiv {\tilde u}_i\equiv {\hat u}_i\ \mbox{modulo}\ (x)^ck[[x,y]] $.  Clearly     $\hat v$ is the limit of $(v_c)_c$  in the $(x)$-adic topology   and belongs to $k[[x]](M\cap k\langle x\rangle^p)$.
\hfill\ \end{proof}

\begin{Remark}{\em  When $p=1$ then the above $M$ is an ideal and we get the so-called (see \cite{Rond1}) strong elimination property of the algebraic power series.}
\end{Remark}
The following proposition is partially contained in \cite[Lemma 4.2]{Rond1}.
\begin{Proposition} \label{p1} Let $M\subset k\langle x\rangle^p$ be a finitely generated $k\langle x\rangle$-submodule and $1\leq r_1< \ldots < r_e\leq n$, $p_1,\ldots,p_e$ be some positive integers such that $p=p_1+\ldots +p_e$. Then
$$T=M\cap (k\langle x_1,\ldots,x_{r_1}\rangle^{p_1}\times \ldots \times k\langle x_1,\ldots,x_{r_e}\rangle^{p_e})$$
is dense in
$${\hat T}=(k[[x]]M)\cap (k[[x_1,\ldots,x_{r_1}]]^{p_1}\times \ldots \times k[[ x_1,\ldots,x_{r_e}]]^{p_e}).$$
 Moreover, if $c\in \bf N$ and   ${\hat v}=\sum_{i=1}^t {\hat u}_i a_i\in {\hat T}$  for some $a_i\in M$, ${\hat u}_i\in k[[x]]$ then there exist $u_{ic}\in k\langle x\rangle$ such that  $ u_{ic}\equiv {\hat u}_i\  \mbox{modulo} \ (x)^ck[[x]]$,  $v_c=\sum_{i=1}^t  u_{ic}a_i\in T$ and
 $\hat v$ is the limit of  $(v_c)_c$ in the $(x)$-adic topology.
\end{Proposition}
\begin{proof}  Apply induction on $e$, the case $e=1$ being
done in Proposition \ref{p}.  Assume that $e>1$. We may reduce to the case when $r_e=n$ replacing $M$ by $M\cap k\langle x_1,\ldots,x_{r_e}\rangle^p$ if $r_e<n$. Let
$$q:\Pi_{i=1}^{p_e}k[[x_1,\ldots,x_{r_i}]]^{p_i}\to \Pi_{i=1}^{p_{e-1}}k[[x_1,\ldots,x_{r_i}]]^{p_i},$$
$$q':\Pi_{i=1}^{p_e}k[[x_1,\ldots,x_{r_i}]]^{p_i}\to   k[[x_1,\ldots,x_{r_e}]]^{p_e}$$
be the canonical projections, ${\hat v}=({\hat v_1},\ldots, {\hat v}_p)\in {\hat T},$ and $M_1=q(M)$.
Assume that  ${\hat v}=\sum_{i=1}^t  {\hat u}_ia_i$ for some ${\hat u}_i\in k[[x]]$, $a_i\in M$.
By induction hypothesis applied to $M_1$, $q(\hat v)$  given $c\in \bf N$ there exists  $ u_{ic}\in k\langle x\rangle$ with $u_{ic}\equiv {\hat u}_i\ \mbox{modulo}\ (x)^ck[[x]]$
 such that $v'_c=\sum_{i=1}^t  u_{ic}q(a_i)\in q(T)$  and $q({\hat v})$ is  the limit of  $(v'_c)_c$ in the $(x)$-adic topology.

Now,  let $v''_c=\sum_{i=1}^t u_{ic}q'(a_i)\in k\langle x_1,\ldots,x_n\rangle^{p_e} $. We have $v''_c\equiv q'({\hat v})\\
 \mbox{modulo}\ (x)^ck[[x]]^{p_e}$. Then  $v_c=(v'_c,v''_c)=\sum_{i=1}^t u_{ic}a_i\in T$, $v_c\equiv {\hat v}\ \mbox{modulo} \
  (x)^ck[[x]]^{p}$  and $\hat v$ is the limit of  $(v_c)_c$ in the $(x)$-adic topology.
\hfill\ \end{proof}

\begin{Corollary} \label{lnes} Theorem \ref{nes} holds when $f$ is linear.
\end{Corollary}
\begin{proof} If $f$ is homogeneous then it is enough
to apply Proposition \ref{p1} for the module $M$ of the solutions of $f$ in $A=k\langle x\rangle$. Suppose that $f$ is not homogeneous, let us say $f$ has the form $g+a_0$ for some system of linear homogeneous polynomials $g\in A[Y]^s$ and $a_0\in A^s$. The proof in this case follows  \cite[page 7]{Rond1}  and we give it here only for the sake of  completeness. Change $f$ by the homogeneous system of linear polynomials ${\bar f}=g+a_0Y_0$ from $A[Y_0,Y]^s$. A nested formal solution $\hat y$ of $\bar f$ in $k[[x]]$ with  ${\hat y}_i\in K[[x_1,\ldots,x_{r_i}]]$, $1\leq i\leq m$ induces a nested formal solution $({\hat y}_0,{\hat y}) $, ${\hat y}_0=1$ of $\bar f$ with $r_0=r_1$. As above, for all $c\in \bf N$ we get a nested algebraic   solution $(y_0,y)$ of $\bar f$ with $y_i\in k\langle x_1,\ldots,x_{r_i}\rangle$ and $y_i\equiv {\hat y}_i\ \mbox{modulo}\ (x)^ck[[x]]$ for all $0\leq i\leq m$. It follows that $y_0$ is invertible and clearly $y_0^{-1}y$ is the wanted  nested algebraic solution of $f$.
\hfill\ \end{proof}

\end{document}